\documentclass[11pt]{amsart}
\usepackage[margin=1in]{geometry}
\usepackage{amssymb,amsfonts,amsmath}
\usepackage{color}
\usepackage{soul}
\usepackage{enumerate}
\usepackage{mathrsfs}
\usepackage{hyperref}
\usepackage[capitalise]{cleveref}
\usepackage{constants} 

\newtheorem{theorem}{Theorem}[section]
\newtheorem{corollary}[theorem]{Corollary}

\newtheorem{proposition}[theorem]{Proposition}
\newtheorem{lemma}[theorem]{Lemma}

\newtheorem{question*}{Question}
\newtheorem{problem*}{Problem}

\theoremstyle{definition}

\newtheorem{example}[theorem]{Example}

\theoremstyle{remark}
\newtheorem*{remark}{Remark}

\theoremstyle{remark}

\numberwithin{equation}{section}

\crefname{figure}{Figure}{Figures}
\theoremstyle{plain}
\newtheorem*{theorem*}{Theorem}
\crefname{theorems}{Theorem}{Theorems}
\crefname{corollaries}{Corollary}{Corollaries}
\newtheorem*{corollary*}{Corollary}
\crefname{corollaries*}{Corollary}{Corollaries}
\crefname{lemma}{Lemma}{Lemmata}
\crefname{proposition}{Proposition}{Propositions}
\crefname{conjectures}{Conjecture}{Conjectures}
\newtheorem*{conjonjecture*}{Conjecture}
\crefname{conjonjectures*}{Conjecture}{Conjectures}
\crefname{definitions}{Definition}{Definitions}
\crefname{hypotheses}{Hypothesis}{Hypotheses}

\newcommand{\Z}{\mathbb{Z}}
\newcommand{\R}{\mathbb{R}}
\newcommand{\Q}{\mathbb{Q}}

\newcommand{\kd}{\mathfrak{d}}

\newcommand{\re}{\textup{Re}}
\newcommand{\im}{\textup{Im}}

\newcommand{\N}{\mathrm{N}}

\newcommand{\Gal}{\mathrm{Gal}}

\makeatletter
\DeclareFontFamily{U}  {MnSymbolF}{}
\DeclareSymbolFont{symbolsMN}{U}{MnSymbolF}{m}{n}
\SetSymbolFont{symbolsMN}{bold}{U}{MnSymbolF}{b}{n}
\DeclareFontShape{U}{MnSymbolF}{m}{n}{
    <-6>  MnSymbolF5
   <6-7>  MnSymbolF6
   <7-8>  MnSymbolF7
   <8-9>  MnSymbolF8
   <9-10> MnSymbolF9
  <10-12> MnSymbolF10
  <12->   MnSymbolF12}{}
\DeclareFontShape{U}{MnSymbolF}{b}{n}{
    <-6>  MnSymbolF-Bold5
   <6-7>  MnSymbolF-Bold6
   <7-8>  MnSymbolF-Bold7
   <8-9>  MnSymbolF-Bold8
   <9-10> MnSymbolF-Bold9
  <10-12> MnSymbolF-Bold10
  <12->   MnSymbolF-Bold12}{}
\DeclareMathSymbol{\tbigtimes}{\mathop}{symbolsMN}{2}
\newcommand*{\bigtimes}{%
  \DOTSB
  \tbigtimes
  \slimits@ 
}
\makeatother

\renewcommand{\tilde}{\widetilde}

\renewcommand{\bar}{\overline}

\newcommand{\Class}{\mathrm{Cl}}

\renewcommand{\hat}{\widehat}

\renewcommand{\epsilon}{\varepsilon}

\newcommand{\kF}{\mathfrak{F}}

\renewcommand{\Im}{\mathrm{Im}}

\newcommand{\cO}{\mathcal{O}}

\newcommand{\kp}{\mathfrak{p}}

\renewcommand{\Re}{\mathrm{Re}}
\newcommand{\reg}{\mathrm{reg}}

\newconstantfamily{abcon}{symbol=c}

\makeatletter
\let\@wraptoccontribs\wraptoccontribs
\makeatother

\title{A zero density estimate for Dedekind zeta functions}

\author{Jesse Thorner}
\address{Department of Mathematics, University of Illinois, Urbana, IL 61801, USA}
\email{jesse.thorner@gmail.com}
\author{Asif Zaman}
\address{Department of Mathematics, University of Toronto, Toronto, Ontario, Canada M5S 2E4}
\email{zaman@math.toronto.edu}

\begin{document}

\begin{abstract}
Given a nontrivial finite group $G$, we prove the first zero density estimate for families of Dedekind zeta functions associated to Galois extensions $K/\Q$ with $\mathrm{Gal}(K/\Q)\cong G$ that does not rely on unproven progress towards the strong form of Artin's conjecture.  We use this to remove the hypothesis of the strong Artin conjecture from the work of Pierce, Turnage-Butterbaugh, and Wood on the average error in the Chebotarev density theorem and $\ell$-torsion in ideal class groups.
\end{abstract}

\maketitle

\section{Statement of the main result}

The study of zeros of $L$-functions associated to families of automorphic representations has stimulated much research over the last century.  Zero density estimates, which show that few $L$-functions within a family can have zeros near the edge of the critical strip, are especially useful. They often allow one to prove arithmetic results that would otherwise require unproven progress towards the generalized Riemann hypothesis (GRH) for the family. Classical triumphs include Linnik's log-free zero density estimate for $L$-functions associated to Dirichlet characters modulo $q$ \cite{Linnik} and the Bombieri--Vinogradov theorem for primes in arithmetic progressions \cite{Bombieri2}.  More recently, Kowalski and Michel \cite{KM} proved a log-free zero density estimate for the $L$-functions of cuspidal automorphic representations of $\mathrm{GL}_m$ over $\Q$ satisfying the generalized Ramanujan conjecture.  Their work was made unconditional by the authors \cite{TZ_GLn}.  These proofs use large sieve inequalities, ensuring that certain Dirichlet polynomials indexed by the representations in the family are small on average.  The large sieve for Dirichlet characters is classical \cite{Bombieri2}.  In \cite{KM,TZ_GLn}, the large sieve inequalities that yield strong zero density estimates near $\re(s)=1$ have relied crucially on the automorphy of the representations in the family so that one can attach to a pair of representations a Rankin--Selberg $L$-function having good analytic properties (e.g., analytic continuation).

Here, we consider certain families where the ability to form Rankin--Selberg $L$-functions with good properties is not yet known because automorphy has not yet been proved.  Let $G$ be a finite group, and let $K$ be a number field, Galois over $\Q$, with $D_K=|\mathrm{disc}(K/\Q)|$ and $\Gal(K/\Q)\cong G$.  Let $\zeta_{K}(s)$ be the Dedekind zeta function of $K$ and $\zeta(s)$ be the Riemann zeta function.  Aramata and Brauer proved that $\zeta_K(s)/\zeta(s)$ is entire, but it is not yet known to be the $L$-function of an automorphic representation except in special cases.  Such an automorphy result would be an immediate consequence of the strong Artin conjecture.  For $\sigma\geq 0$ and $T\geq 0$, define
\begin{equation}
\label{eqn:NsigmaT}
N_{K}(\sigma,T) := \#\{\beta+i\gamma \in \mathbb{C}\colon \zeta_K(\beta+i\gamma)/\zeta(\beta+i\gamma) = 0,~\beta \geq \sigma,~|\gamma| \leq T\}.
\end{equation}
Note that $\zeta_K(s)/\zeta(s)$ satisfies GRH if and only if $N_K(\sigma,T)=0$ for all $\sigma>\frac{1}{2}$ and all $T\geq 0$.
\begin{theorem} \label{thm:ZDE-Artin}
Let $G$ be a nontrivial finite group.  Let $\mathfrak{F}_G$ be a family of number fields $K$ that are Galois over $\Q$ with $\Gal(K/\Q)\cong G$, and let $\mathfrak{F}_G(Q)=\{K\in\mathfrak{F}_G\colon D_K\leq Q\}$ be nonempty.  Define
\begin{equation}
\label{def:intersection-multiplicity}
	\mathfrak{m}_{\kF_G}(Q) := \max_{ K' \in \mathfrak{F}_G(Q)} \#\{ K \in \mathfrak{F}_G(Q)\colon K \cap K' \neq \Q \}. 
\end{equation}
If $\sigma\geq 0$ and $Q,T\geq 2$, then
	\[
	\sum_{K \in \mathfrak{F}_G(Q)} N_{K}(\sigma,T) \ll_{|G|} \mathfrak{m}_{\kF_G}(Q) (QT)^{10^7 |G|^3(1-\sigma)}(\log QT)^{|G|^2}.
	\]
\end{theorem}

\begin{remark}
By assembling the ideas going into the proof of \cite[Theorem 5.5]{PTW} (which build on \cite{KM}), one can extract from the work of Pierce, Turnage-Butterbaugh, and Wood \cite{PTW} a variant of \cref{thm:ZDE-Artin} whose proof relies crucially on the strong Artin conjecture.
\end{remark}

Along with avoiding automorphy assumptions and possessing a relatively clean proof, the most interesting part of \cref{thm:ZDE-Artin} is the presence of the ``intersection multiplicity'' $\mathfrak{m}_{\mathfrak{F}_G}(Q)$, which seems to be new to the literature.  It arises naturally in our proof, and it is at least 1 when $\#\mathfrak{F}_G(Q)\geq 1$.  Its asymptotic estimation is a purely arithmetic problem determined by $G$ and $\mathfrak{F}_{G}$.

\section{Applications}
\label{sec:cheb_average}

Let $G$ be a finite group, $K/\Q$ be a Galois extension of number fields with $\mathrm{Gal}(K/\Q)\cong G$, and
\begin{align*}
\pi_C(x,K/\Q):=\#\Big\{p\colon p\nmid D_K,~\Big[\frac{K/F}{p}\Big]=C,~ p\leq x\Big\},
\end{align*}
where $C\subseteq G$ is a given conjugacy class and the Artin symbol $[\frac{K/F}{p}]$ denotes the conjugacy class of Frobenius automorphisms attached to the prime ideals of $K$ lying over $p\nmid D_K$.  Lagarias and Odlyzko \cite{LO} proved that under GRH for $\zeta_{K}(s)$, we have an effective form of the Chebotarev density theorem that counts prime ideals of small norm with a given Artin symbol:
\begin{equation}
\label{eqn:GRH_range_of_x}
\Big|\pi_C(x,K/F)-\frac{|C|}{|G|}\int_2^x\frac{dt}{\log t}\Big|\ll\frac{|C|}{|G|}x^{\frac{1}{2}}\log(D_{{K}}x^{[{K}:\Q]}),\quad x\gg (\log D_{{K}})^{2}(\log\log D_K)^4.
\end{equation}

Without GRH, one might hope that something similar to \eqref{eqn:GRH_range_of_x} holds on average over a set $\kF_G$ of number fields $K$ that are Galois over $\Q$ with $\Gal(K/\Q)\cong G$.  If $K/\Q$ ranges over abelian extensions, then class field theory implies that the  Artin $L$-functions of $K/\Q$ are in fact Dirichlet $L$-functions, which brings us to the setting of \cite{Bombieri2}.   A recent breakthrough of Pierce, Turnage-Butterbaugh, and Wood \cite[Theorems 1.1 and 1.4]{PTW} allows one to control the average size of the error term in the Chebotarev density theorem across families of extensions $K/\Q$ with a \textit{fixed nonabelian} Galois group, subject to hypotheses that do not include GRH.  We describe part of their work in an exemplary case, namely $G=S_n$ with $n \geq 5$.  Assume the strong Artin conjecture for $S_n$ (that all irreducible Galois representations over $\Q$ with image isomorphic to $S_n$ are in fact cuspidal automorphic representations).  For a number field $k$, denote by $\widetilde{k}$ the Galois closure over $\Q$. Define
\begin{equation}
\label{eqn:PTW-S_m}
\mathscr{F}_n(D):=\{k\colon [k:\Q]=n,~\Gal(\widetilde{k}/\Q)\cong S_n,~\textup{$D_k$ squarefree},~D_k\leq D\}.
\end{equation}
 If there exists a constant $\delta_n>0$ such that
\begin{equation}
\label{eqn:PTW1}
\max_{D'\leq D}\#\{k\in\mathscr{F}_n(D)\colon D_k = D'\}\ll_{n} D^{-\delta_n}\#\mathscr{F}_n(D),
\end{equation}
then for every $A > 0$ there exists a constant $c=c_{A,n}>0$ such that the number fields $k$ with $[k:\Q]=n$, squarefree absolute discriminant $D_k\leq D$, and $\Gal(\widetilde{k}/\Q)\cong S_n$ satisfy
\begin{equation}
\label{eqn:PTW2}
\Big|\pi_C(x,\widetilde{k}/\Q)-\frac{|C|}{|G|}\int_2^x\frac{dt}{\log t}\Big|\leq \frac{|C|}{|G|}\frac{x}{(\log x)^A},\qquad x\gg_{A,n}(\log D_k)^{c(\log\log D_k)^{\frac{2}{3}+o(1)}}
\end{equation}
with $O_{n}(D^{-\delta_n}\#\mathscr{F}_n(D))$ exceptions.  Thus, subject to well-believed conjectures not including GRH, they show that for the Galois closures of most degree $n$ $S_n$-fields of squarefree discriminant, one obtains highly accurate counts for primes with a given Artin symbol.  When $G$ is the dihedral group $D_q$ ($q\geq 3$ prime) or a small group (e.g., $S_3$, $S_4$, or $A_4$), their results are unconditional because they prove analogues of \eqref{eqn:PTW1}, and the strong Artin conjecture holds for these groups.

The study of $\ell$-torsion in class groups of number fields helps motivate the work in \cite{PTW}.  It follows from work of Ellenberg and Venkatesh \cite[Lemma 2.3]{EV} that if $k$ is a number field, $\ell\geq 2$, and there exists a constant $\delta>0$ such that there are $M$ primes $p\leq D_k^{1/(2\ell([k:\Q]-1))-\delta}$ that split completely in $k$, then for all $\eta>0$, the $\ell$-torsion subgroup $\mathrm{Cl}_k[\ell]$ of the ideal class group $\mathrm{Cl}_k$ of $k$ satisfies $|\mathrm{Cl}_k[\ell]|\ll_{[k:\Q],\ell,\eta}D_k^{1/2+\eta}M^{-1}$.  Primes that split completely in the Galois closure $\tilde{k}$ (over $\Q$) also split completely in $k$, so if \eqref{eqn:GRH_range_of_x} holds for $\widetilde{k}$, then
\begin{equation}
	\label{eqn:GRH_ell_torsion}
	|\Class_{k}[\ell]|\ll_{[k:\Q],\ell,\eta}D_k^{\frac{1}{2}-\frac{1}{2\ell([k:\Q]-1)}+\eta}.
\end{equation}
This improves on the trivial Minkowski bound $O_{[k:\Q]}(D_k^{1/2}(\log D_k)^{[k:\Q]+1})$.  Even though \eqref{eqn:PTW2} does not provide a range of $x$ as strong as \eqref{eqn:GRH_range_of_x}, \eqref{eqn:PTW2} suffices to deduce \eqref{eqn:GRH_ell_torsion}.

Our next result is a natural application of \cref{thm:ZDE-Artin} that does not rely on unproven progress towards the strong Artin conjecture.

\begin{theorem}
\label{thm:main-Artin}
	Let $G$ be a nontrivial finite group, $\epsilon>0$, and $Q\geq 3$.  Let $\mathfrak{F}_G$ and $\mathfrak{m}_{\mathfrak{F}_G}(Q)$ be as in \eqref{def:intersection-multiplicity}.  For all except $O_{|G|,\epsilon}(\mathfrak{m}_{\mathfrak{F}_G}(Q)Q^{\epsilon})$ number fields $K\in\mathfrak{F}_G(Q)$, the following results hold.
\begin{enumerate}[(i)]
	\item If $C\subseteq G$ is a conjugacy class and $x\geq (\log D_K)^{10^9 |G|^3/\epsilon}$, then
		\[
		\Big|\pi_C(x,K/\Q)-\frac{|C|}{|G|}\int_2^x\frac{dt}{\log t}\Big|\ll_{|G|}x\exp\Big(-\frac{1}{29}\Big(\frac{\log x}{|G|}\Big)^{1/2}\Big).
		\]
	\item Let $\ell\geq 2$ be an integer and $\eta>0$.  If $k\subseteq K$ is a subfield such that $\widetilde{k}=K$, then \eqref{eqn:GRH_ell_torsion} holds.
\end{enumerate}
\end{theorem}
\begin{proof}
Part (i) follows from \cref{prop:FlexiError-Li,prop:large-ZFR} below.   Part (ii) follows from Part (i) by the discussion preceding \eqref{eqn:GRH_ell_torsion}.
\end{proof}

\begin{remark}
\cref{thm:main-Artin} is not vacuous when there exists a constant $\delta=\delta(|G|)>0$ such that
	\begin{equation}
	\label{eqn:condition1}
	\mathfrak{m}_{\kF_G}(Q) \ll_{|G|}Q^{-\delta}\#\mathfrak{F}_G(Q).
	\end{equation}
The condition \eqref{eqn:condition1} implicitly requires $\#\mathfrak{F}_G(Q) \gg_{|G|} Q^{\delta}$ since $\mathfrak{m}_{\mathfrak{F}_G}(Q)\geq 1$.  In general, the condition \eqref{eqn:PTW1} is at least as restrictive as the condition \eqref{eqn:condition1}.
\end{remark}

We present some examples of unsolvable groups $G$ and families $\mathfrak{F}_G$ for which \eqref{eqn:condition1} holds with an explicit $\delta$.  The strong Artin conjecture is not known in these settings, so these examples are not accessible by the methods in \cite{PTW}.  Even for the groups and families considered in \cite{PTW}, \cref{thm:main-Artin} gives an effective Chebotarev density theorem with a range of $x$ that is superior to the range in \eqref{eqn:PTW2} produced by the methods in \cite{PTW}.  We will use the following result:  If $r\geq 2$, $G$ is a transitive subgroup of the symmetric group $S_r$, and $k$ is a number field whose Galois closure $\tilde{k}$ over $\Q$ has Galois group isomorphic to $G$, then we have the bounds $D_k^{|G|/r}\ll_{r} D_{\tilde{k}}\ll_{r} D_k^{|G|/2}$.

\begin{example}
For $n\geq 5$, let $G$ be the alternating group $A_n$ of order $\frac{n!}{2}$ and
	\[
	\mathcal{F}_G(D)=\{k\colon [k:\Q]=n,~\mathrm{Gal}(\widetilde{k}/\Q)\cong G,~D_k\leq D\},\qquad \mathfrak{F}_G=\{\widetilde{k}\colon k\in \mathcal{F}_G(Q^{2/|G|})\}.
	\]
If $K,K'\in\mathfrak{F}_G$ and $K\cap K'\neq\Q$, then $K\cap K'$ is Galois over $\Q$.  Thus, there exist nontrivial normal subgroups $N\subseteq\Gal(K/\Q)$ and $N'\subseteq\Gal(K'/\Q)$ such that $K^N=K\cap K'=(K')^{N'}$.  Since $\Gal(K/\Q)\cong \Gal(K'/\Q)\cong A_n$ and $n\geq 5$, these Galois groups are simple.  The condition $K\cap K'\neq\Q$ is therefore equivalent to $K=K'$, hence $\mathfrak{m}_{\mathfrak{F}_G}(Q) = 1$ for all $Q\geq 1$.   Landesman, Lemke Oliver, and Thorne \cite[Theorem 1.1]{LLOT} proved that $\#\mathcal{F}_G(D) \gg_n D^{1/30}$, so $\#\mathfrak{F}_G(Q)\gg Q^{1/(15|G|)}$.  Thus, \cref{thm:main-Artin} holds with an exceptional set of size $O_{n,\epsilon}(Q^{\epsilon})$, which is nontrivial when $0<\epsilon<1/(15|G|)$.  This appears to be the first unconditional instance where infinitely many extensions of arbitrarily large degree and unsolvable Galois closure over $\Q$ have non-trivial bounds on $\ell$-torsion in their class groups with $\ell\geq 3$.  The argument applies with the same conclusions when $G$ is any finite simple group, provided that there exists a constant $\delta=\delta(|G|)>0$ such that $\#\mathfrak{F}_G(Q)\gg_{|G|}Q^{\delta}$.
\end{example}
	
\begin{example}
For $n\geq 5$, let $G$ be the symmetric group $S_n$ of order $n!$.  Recall \eqref{eqn:PTW-S_m}, and define $\mathfrak{F}_{G}=\{\tilde{k}\colon k\in\mathscr{F}_n(Q^{2/|G|})\}$.  It follows from the setup in \cite[Section 5]{PTW} that $\mathfrak{m}_{\mathfrak{F}_{G}}(Q)$ is bounded by \eqref{eqn:PTW1}.  If there exist constants $c_n>0$ and $0<\varpi_n<1$, depending at most on $n$, such that
	\[
	\#\mathscr{F}_n(D)=c_n D+O_{n}(D^{\varpi_n}),
	\]
	then $\#\mathfrak{F}_{G}(Q)\gg_{n}Q^{2/|G|}$ and $\mathfrak{m}_{\mathfrak{F}_{G}}(Q)\ll_{|G|}Q^{2\varpi_n/|G|}$.  Therefore, \cref{thm:main-Artin} holds with an exceptional set of size $O_{n,\epsilon}(Q^{2(\varpi_n-1)/|G|+\epsilon}\#\mathfrak{F}_{G}(Q))$.  When $n=5$, it follows from work of Bhargava \cite{Bhargava} and Shankar and Tsimerman \cite{ShankarTsimerman} that any $\varpi_5>\frac{199}{200}$ is permissible.
\end{example}


\subsection*{Outline of the paper} 
\cref{sec:Artin} reviews Artin $L$-functions. \cref{sec:prelim_large_sieve} develops a large sieve inequality for Dedekind zeta functions that we use to prove  \cref{thm:ZDE-Artin} in \cref{sec:ZDE}. In \cref{sec:avg_Chebotarev_error}, we use \cref{thm:ZDE-Artin} to prove \cref{thm:main-Artin}.

\subsection*{Acknowledgements}

We thank Kannan Soundararajan for helpful discussions, Nicholas Lai and Lior Silberman for showing us \cite{Brauer}, and the anonymous referee for many helpful comments.  Work began while the authors were at Stanford University.  JT was partially supported by an NSF Postdoctoral Fellowship.  AZ was partially supported by an NSERC fellowship.

 \section{Artin $L$-functions}
\label{sec:Artin}

\subsection{Preliminaries}
\label{subsec:Artin}

We briefly recall the definition of an Artin $L$-function from \cite[Chapter 2, Section 2]{MurtyMurty}.  Let $K/\Q$ be a Galois extension of number fields with Galois group isomorphic to $G$. Let $\cO_K$ be the ring of integers of $K$.  For each prime $p$, and a prime ideal $\mathfrak{p}$ of $K$ lying above $p$, we define the decomposition group $D_{\mathfrak{p}}$ to be $\mathrm{Gal}(K_{\mathfrak{p}}/\Q_{p})$, where $K_{\mathfrak{p}}$ (resp. $\Q_{p}$) is the completion of $K$ (resp. $\Q$) at $\mathfrak{p}$ (resp. $p$).  We have a map $D_{\mathfrak{p}}\to\mathrm{Gal}(k_{\mathfrak{p}}/k_{p})$ (the Galois group of the residue field extension), which is surjective by Hensel's lemma.  The kernel of this map is the inertia group $I_{\mathfrak{p}}$, so we have the exact sequence $1\to I_{\mathfrak{p}}\to D_{\mathfrak{p}}\to\mathrm{Gal}(k_{\mathfrak{p}}/k_{p})\to 1$.  The group $\mathrm{Gal}(k_{\mathfrak{p}}/k_{p})$ is cyclic with generator $x\mapsto x^{p}$, where $p$ is the cardinality of $k_{p}$.  We can choose an element $\sigma_{\mathfrak{p}}\in D_{\mathfrak{p}}$ whose image in $\mathrm{Gal}(k_{\mathfrak{p}}/k_{p})$ is this generator.  We call $\sigma_{\mathfrak{p}}$ a Frobenius element at $\mathfrak{p}$; it is well-defined modulo $I_{\mathfrak{p}}$.  If $p$ is unramified and $\kp$ lies above $p$, then $I_{\mathfrak{p}}$ is trivial.

Let $\rho:G\to\mathrm{GL}_n(\mathbb{C})$ be a representation of $G$, $\chi$ denote its character, and $q_{\chi}$ denote its conductor.  Let $V$ be the underlying complex vector space on which $\rho$ acts, and let $V^{I_{\mathfrak{p}}}$ be the subspace of $V$ on which $I_{\mathfrak{p}}$ acts trivially.  If $p$ is prime and $\kp$ lies above $p$, then the eigenvalues of $\rho(\sigma_{\kp})|_{V^{I_{\kp}}}$ depend only on $\chi$ and $p$.  Therefore, we write the eigenvalues as $\alpha_{1,\chi}(p),\ldots,\alpha_{n,\chi}(p)$, each of which has modulus at most 1.  If $p$ is unramified, then each $\alpha_{j,\chi}(p)$ has modulus equal to 1.  We now define
\[
L_{p}(s,\chi)=\prod_{j=1}^n\Big(1-\frac{\alpha_{j,\chi}(p)}{p^{s}}\Big)^{-1},\qquad L(s,\chi) = \prod_{p}L_p(s,\chi)=\sum_{n=1}^{\infty}\frac{\lambda_{\chi}(n)}{n^s}.
\]
Since $|\alpha_{j,\chi}(p)|\leq 1$ always, $L(s,\chi)$ converges absolutely for $\re(s)>1$.

If $\Gamma_{\R}(s) = \pi^{-s/2}\Gamma(s/2)$ and $L_{\infty}(s,\chi) = 
\Gamma_{\R}(s)^{a(\chi)} \Gamma_{\R}(s+1)^{\chi(1)-a(\chi)}$, where $a(\chi)$ is the dimension of the $+1$ eigenspace of complex conjugation, then $\Lambda(s,\chi)=  (q_{\chi})^{s/2}L(s,\chi)L_{\infty}(s,\chi)$ has a meromorphic continuation to $\mathbb{C}$, and there exists $W(\chi) \in \mathbb{C}$ of modulus one such that $\Lambda(s,\chi) = W(\chi)\overline{\Lambda(1-\bar{s},\chi)}$ for all $s\in\mathbb{C}$ at which $\Lambda(s,\chi)$ is holomorphic. Conjecturally, $\Lambda(s,\chi)$ is entire when $\chi$ is nontrivial. When $\chi$ is trivial, we have $L(s,\chi)=\zeta(s)$.

\subsection{The Aramata--Brauer theorem}

For a finite group $G$, let $\reg_G$ (resp. $1_G$) denote the character of the regular (resp. trivial) representation of $G$. For a Galois extension $K/\Q$ of number fields, let $\chi_K$ denote the character of $\Gal(K/\Q)$ defined by $\chi_K = \reg_{\Gal(K/\Q)} - 1_{\Gal(K/\Q)}$.  Its associated Artin $L$-function  is given by $L(s,\chi_K) = \zeta_{K}(s)/\zeta(s)$.  Per \cref{subsec:Artin}, there exist complex numbers $\alpha_{j,K}(p)$ with modulus at most 1 such that if $\re(s)>1$, then
\begin{equation}
	\label{eqn:Artin-localroots}
	L(s,\chi_K) =  \prod_{p} \prod_{j=1}^{[K:\Q]-1}  \Big(1 - \frac{\alpha_{j,K}(p)}{p^s}\Big)^{-1} = \sum_{n=1}^{\infty} \frac{\lambda_{K}(n)}{n^{s}}.  
\end{equation}
The Dirichlet series coefficients $\lambda_{K}(n)$ are defined in terms of the local roots $\alpha_{j,K}(p)$ via this identity.  Since $\chi_K$ is a real-valued valued character, we have that $\lambda_K(n)\in\R$ for all $n\geq 1$.   Also, since $D_K$ is the conductor of $\zeta_{K}(s)$ over $\Q$, it follows that $q_{\chi_K}=D_K$.  

\begin{theorem}[Aramata--Brauer]
\label{thm:Aramata_Brauer}
	The Artin $L$-function $L(s,\chi_K)$ is entire of order one. 
\end{theorem}

\subsection{Composite fields}
Let  $K/\Q$ (resp. $K'/\Q$)  be a Galois extension of number fields with Galois group $G$ (resp. $G'$) and ring of integers $\cO_K$ (resp. $\cO_{K'}$). Write $m= [K:\Q]-1$ and $m'=  [K':\Q]-1$.  Assume that $K \cap K' = \Q$.  Since $K$ and $K'$ are Galois over $\Q$, our assumption implies that $K$ and $K'$ are linearly disjoint over $\Q$.  This condition yields a natural isomorphism
	\[
	\mathrm{Gal}(K'K/\Q) \mathop{\longrightarrow}^{\cong} \mathrm{Gal}(K/\Q) \times \mathrm{Gal}(K'/\Q) = G \times G'
	\]
	which sends $\sigma \mapsto (\sigma|_{K},  \sigma|_{K'})$. Moreover, via this isomorphism, every irreducible character of $\mathrm{Gal}(K'K/\Q)$ is an (external) tensor product of the shape $\chi \otimes \chi'$, where $\chi$ (resp. $\chi'$) is the character of an irreducible representation of $G$ (resp. $G'$). Thus, as  characters,
	\[
	\reg_{G \times G'}   = \reg_{G}  \otimes \reg_{G'}  = 1_{G \times G'} +  \chi_{K} \otimes 1_{G'}  + 1_{G} \otimes \chi_{K'}   + \chi_{K} \otimes \chi_{K'}.
	\]
In terms of Artin $L$-functions, this yields
\begin{equation}
	\zeta_{K'K}(s) = \zeta(s) L(s,\chi_{K}) L(s,\chi_{K'})   L(s,\chi_K \otimes \chi_{K'}).  
\label{eqn:Artin-decomposition}
\end{equation}
The Artin $L$-function $L(s,\chi_K \otimes \chi_{K'})$ is defined by the character $\chi_K \otimes \chi_{K'}$ of $\Gal(K'K/\Q)$.  The following lemma  provides  useful bounds for the conductor $q_{\chi_K\otimes\chi_{K'}}$.

\begin{lemma} \label{lem:Artin-BH}
	If $K/\Q$ and $K'/\Q$ are Galois and $K \cap K' = \Q$, then $D_{K'K}|D_{K}^{[K':\Q]}  D_{K'}^{[K:\Q]}$. 
\end{lemma}
\begin{proof}
We use \cite[Part III, Section 2]{Neukirch}. As $K \cap K' = \Q$, the relative discriminant $\kd_{K'K/K}\subseteq\cO_{K}$ satisfies $D_{K'K} = D_{K}^{[K':\Q]} \N_{K/\Q}  \kd_{K'K/K}$. By definition, $\kd_{K'K/K}$ is generated by all bases of $K'K/K$ contained in $\cO_{K'K}$.  Since $K  \cap K' = \Q$, any basis of $K'/\Q$ contained in $\cO_{K'}$ is a basis for $K'K/K$ contained in $\cO_{K'K}$. Thus, $\kd_{K'K/K}$ contains $D_{K'} \cO_{K}$, so $\kd_{K'K/K}$  divides $D_{K'} \cO_{K}$. Since $K \cap K' = \Q$, we have $\N_{K/\Q} (D_{K'} \cO_{K}) = D_{K'}^{[K:\Q]}$ so  we conclude that $D_{K'K}$ divides $D_{K}^{[K':\Q]} D_{K'}^{[K:\Q]}$, as desired. 
\end{proof}
	
	\begin{proposition}
	\label{prop:AramataBrauer-extension}
	Let  $K/\Q$ (resp. $K'/\Q$) be Galois extensions of number fields with $[K:\Q]=m+1$ and (resp. $[K':\Q]=m'+1$). If $K \cap K' =  \Q$, then the Artin $L$-function $ L(s,\chi_{K} \otimes \chi_{K'})=\zeta_{K'K}(s)\zeta(s)\zeta_{K}(s)^{-1}\zeta_{K'}(s)^{-1}$ is entire. Furthermore, its Artin conductor divides $D_{K}^{m'} D_{K'}^{m}$.
	\end{proposition}

\begin{proof}
Brauer \cite{Brauer} proved that $L(s,\chi_K\otimes\chi_{K'})$ is entire, and the expression for $L(s,\chi_K \otimes \chi_{K'})$ as a ratio of Dedekind zeta functions  follows from \eqref{eqn:Artin-decomposition}.  Therefore, the Artin conductor of $L(s,\chi_K\otimes\chi_{K'})$ is $D_{K'K}/(D_K D_{K'})$, which divides $D_K^{m'}D_{K'}^{m}$ by \cref{lem:Artin-BH}.
\end{proof}

\section{A large sieve inequality for Artin representations}
\label{sec:prelim_large_sieve}

Let $G$ be a nontrivial finite group, and let $\mathfrak{F}_G$ be a family of number fields $K$, Galois over $\Q$, with $\Gal(K/\Q)\cong G$.  Let $Q\geq 1$ and $\mathfrak{F}_G(Q)=\{K\in\mathfrak{F}_G\colon D_K\leq Q\}$.  Our main result for this section is a large sieve inequality for the Dirichlet coefficients $\lambda_K(n)$ as $K\in\mathfrak{F}_G(Q)$ varies.  Its key features are the absence of unproven hypotheses towards the strong Artin conjecture (even though it is not yet known in full generality for $L(s,\chi_K)$ or $L(s,\chi_K\otimes\chi_{K'})$) and the role of $\mathfrak{m}_{\mathfrak{F}_G}(Q)$ in \eqref{def:intersection-multiplicity}.
\begin{theorem}
\label{thm:pre_large_sieve}
	Let $b:\Z\to\mathbb{C}$ be a function, and let $G$ be a nontrivial finite group of order $m+1$.  There exists a constant $B_m>0$, depending only on $m$, such that if $Q,T\geq 1$, $x>1$ and $\epsilon>0$, then
	\begin{multline*}
	\sum_{K\in\mathfrak{F}_G(Q)}\Big|\sum_{\substack{n \in(x,xe^{1/T}]}}\lambda_{K}(n)b(n)\Big|^2\\
	\ll_{m,\epsilon}\Big(\mathfrak{m}_{\kF_G}(Q)\frac{x(\log x)^{m^{2}-1}}{T}+\sqrt{x}(\log x)^{B_m}Q^{\frac{m}{2}+1+\epsilon}T^{\frac{m^2}{4}+\epsilon}\Big)\sum_{n \in(x,xe^{1/T}]}|b(n)|^2.
	\end{multline*}
\end{theorem}

Let $\phi$ be a smooth test function which is supported on a compact subset of $[-2,2]$, and let
\begin{equation}
\label{eqn:mellin}
\widehat{\phi}(s) = \int_{\R}\phi(y)e^{sy}dy
\end{equation}
be its Laplace transform.  Then $\widehat{\phi}(s)$ is entire, and for each integer $j\geq 0$, we have that
\begin{equation}
\label{eqn:test_fcn_bounds}
\widehat{\phi}(s)\ll_{\phi,j}e^{2|\mathrm{Re}(s)|}|s|^{-j}.
\end{equation}
Let $T\geq 1$. By Laplace inversion, if $x>0$ and $c\in\R$, then
\[
\phi(T\log x)=\frac{1}{2\pi i T}\int_{c-i\infty}^{c+i\infty}\widehat{\phi}\Big(\frac{s}{T}\Big)x^{-s}ds.
\]

Our proof of \cref{thm:pre_large_sieve} requires control over the sums
\begin{equation}
\label{eqn:K'K}
\sum_{n\geq 1}\lambda_{K}(n)\lambda_{K'}(n)\phi\Big(T\log\frac{n}{x}\Big).
\end{equation}
When $K\cap K'=\Q$, $\lambda_K(n)\lambda_{K'}(n)$ is related to the $n$-th Dirichlet coefficient of $L(s,\chi_K\otimes\chi_{K'})$.
\begin{lemma}
	\label{lem:DK}
	Let $G$ be a finite group of order $m+1\geq 2$ and $\epsilon>0$.  Let $K$ and $K'$ be Galois extensions of $\Q$ with $\Gal(K/\Q),\Gal(K'/\Q)\cong G$.  If $K\cap K'=\Q$, then there exists a function $H(s;K,K')$, holomorphic in the region $\re(s)>\frac{1}{2}$, such that
	\[
	\sum_{n=1}^{\infty}\frac{\lambda_K(n)\lambda_{K'}(n)}{n^s} = L(s,\chi_K\otimes\chi_{K'})H(s;K,K').
	\]
	Moreover, there is a constant $B_m>0$ such that $|H(s;K,K')|\ll_{m,\epsilon}(D_K D_{K'})^{\frac{\epsilon}{4}}(\re(s)-\frac{1}{2})^{-B_m}$.
\end{lemma}
\begin{proof}
If $K\cap K'=\Q$, then per \cref{lem:Artin-BH} and \cref{prop:AramataBrauer-extension}, the $L$-function $L(s,\chi_K\otimes\chi_{K'})$ is an entire Artin $L$-function whose ramified primes necessarily divide $D_K D_{K'}$.  For each prime $p|D_{K}D_{K'}$, there exist complex numbers $\alpha_{j,j',K\times K'}(p)$ for $1\leq j\leq m$ and $1\leq j'\leq m'$, each with modulus at most 1, such that if $\re(s)>1$, then
\[
L(s,\chi_K\otimes\chi_{K'})=\Big(\prod_{p|D_K D_{K'}}\prod_{j=1}^m \prod_{j'=1}^{m'}\Big(1-\frac{\alpha_{j,j',K\times K'}(p)}{p^s}\Big)^{-1}\Big)\prod_{p\nmid D_K D_{K'}}\prod_{j=1}^m \prod_{j'=1}^{m'}\Big(1-\frac{\alpha_{j,K}(p)\alpha_{j',K'}(p)}{p^s}\Big)^{-1}.
\]
The proof now proceeds exactly according to the recipe in \cite[Proposition 2]{DK}.
\end{proof}

\begin{lemma}
	\label{lem:convexity}
	Let $\sigma\geq 0$, $t\in\R$, and $\epsilon>0$.  Under the hypotheses of \cref{lem:DK}, we have $|L(\sigma+it,\chi_K\otimes\chi_{K'})|\ll_{m,\epsilon} (D_K^m D_{K'}^m (1+|t|)^{m^2})^{\max\{\frac{1-\sigma}{2},0\}+\frac{\epsilon}{4}}$.
\end{lemma}
\begin{proof}
	Since $|\alpha_{j,K}(p)|$, $|\alpha_{j',K'}(p)|$, and $|\alpha_{j,j',K\times K'}(p)|$ are at most 1, this follows from \cite[(5.20)]{IK}.
\end{proof}

\begin{lemma}
\label{lem:local_density}
Let $T,x\geq 1$ and $\epsilon>0$.  If $K\cap K'=\Q$ and $K,K'\in\mathfrak{F}_G(Q)$, then
	\begin{align*}
	\Big|\sum_{n\geq 1} \lambda_{K}(n)\lambda_{K'}(n)\phi\Big(T\log\frac{n}{x}\Big)\Big|\ll_{m,\phi,\epsilon}\sqrt{x}(\log x)^{B_m} Q^{\frac{m}{2}+\epsilon}  T^{\frac{m^2}{4}+\epsilon}.
	\end{align*}
\end{lemma}
\begin{proof}
It suffices to take $0<\epsilon<1$.  If $K\cap K'=\Q$, then by \cref{lem:DK}, the desired sum equals
	\[
		\Big|\frac{1}{2\pi i T}\int_{\frac{1}{2}+\frac{1}{4\log(e x)}-i\infty}^{\frac{1}{2}+\frac{1}{4\log(ex)}+i\infty}L(s,\chi_K\otimes\chi_{K'})H(s;K,K')\widehat{\phi}\Big(\frac{s}{T}\Big) x^s ds\Big|.
	\]
By \cref{lem:DK,lem:convexity} and \eqref{eqn:test_fcn_bounds}, the integral is
	\begin{align*}
	&\ll_{m,\epsilon} \frac{\sqrt{x}(\log x)^{B}}{T}Q^{\frac{m}{2}+\epsilon} \int_{-\infty}^{\infty}(2+|t|)^{\frac{m^2}{4}+\epsilon}\Big|\widehat{\phi}\Big(\frac{1}{T}\Big(\frac{1}{2}+\frac{1}{4\log(ex)}+it\Big)\Big)\Big|dt\\
	&\ll_{m,\phi,\epsilon} \frac{\sqrt{x}(\log x)^{B}}{T}Q^{\frac{m}{2}+\epsilon} \int_{-\infty}^{\infty}(2+|t|)^{\frac{m^2}{4}+\epsilon}\min\Big\{1,\frac{T^{\lfloor\frac{m^2}{4}+2\rfloor+1}}{(2+|t|)^{\lfloor\frac{m^2}{4}+2\rfloor+1}}\Big\}dt,
	\end{align*}
which is bounded as claimed.
\end{proof}

Assuming the strong Artin conjecture, if $K \cap K' \neq \Q$, then the Artin $L$-function whose analytic properties control the partial sums of $\lambda_{K}(n)\lambda_{K'}(n)$ will have a pole at $s=1$ whose order could be as small as 1 and as large as $m^2$.  Therefore, we cannot do much better than the following unconditional result.

\begin{lemma}
	\label{lem:trivial_bound_coeffs}
	Let $T,x\geq 1$.  If $K,K'\in\mathfrak{F}_G(Q)$ and $|G|=m+1$, then
	\[
\sum_{n\geq 1}|\lambda_{K}(n)\lambda_{K'}(n)| \phi\Big(T\log\frac{n}{x}\Big)\ll_{m,\phi} \frac{x}{T}(\log x)^{m^{2}-1}+\sqrt{x}T^{\frac{m^2}{4}}.
	\]
\end{lemma}
\begin{proof}
Since $|\alpha_{j,K}(p)|$ and $|\alpha_{j',K'}(p)|$ are at most 1, we have $|\lambda_{K}(n)\lambda_{K'}(n)| \leq d_{m^2}(n)$, where $d_{m^2}(n)$ is the $n$-th Dirichlet coefficient of $\zeta(s)^{m^2}$.  Therefore, the desired sum is at most
\[
\frac{1}{2\pi i T}\int_{3-i\infty}^{3+i\infty}\zeta(s)^{m^2}\hat{\phi}(s)ds\ll_{m,\phi}\frac{x(\log x)^{m^{2}-1}}{T}+\frac{\sqrt{x}}{T}\int_{\frac{1}{2}-i\infty}^{\frac{1}{2}+i\infty}|\zeta(s)|^{m^2}|\widehat{\phi}(s)|\cdot|ds|,
\]
which is bounded as claimed using the convexity bound $|\zeta(\frac{1}{2}+it)|\ll(|t|+1)^{1/4}$ and \eqref{eqn:test_fcn_bounds}.
\end{proof}

We are now in position to prove our large sieve inequality.

\begin{proof}[Proof of \cref{thm:pre_large_sieve}]
It suffices to assume that $\sum_{n\in(x,xe^{1/T}]}|b(n)|^2=1$.  By the duality principle for bilinear forms, we have
\begin{equation}
\label{eqn:ratio_11}
\sum_{K\in\mathfrak{F}_G(Q)}\Big|\sum_{\substack{n \in(x,xe^{1/T}]}}\lambda_{K}(n)b(n)\Big|^2\leq \sup_{\|\beta\|_2=1}\sum_{n \in(x,xe^{1/T}] }\Big|\sum_{K\in\mathfrak{F}_G(Q)}\lambda_{K}(n)\beta(K)\Big|^2,
\end{equation}
where $\beta$ ranges over the functions from $\mathfrak{F}_G(Q)$ to $\mathbb{C}$ that satisfy $\sum_{K\in\mathfrak{F}_G(Q)}|\beta(K)|^2=1$.  Fix a nonnegative smooth function $\phi$, supported on a compact subset of $[-2,2]$, such that $\phi(T\log\frac{t}{x})$ is a pointwise upper bound for the indicator function of $(x,xe^{1/T}]$.  Then \eqref{eqn:ratio_11} is
\begin{equation}
\label{eqn:ratio_3}
\leq \sup_{\|\beta\|_2=1}\sum_{n\geq 1}\Big|\sum_{K\in\mathfrak{F}_G(Q)}\lambda_{K}(n)\beta(K)\Big|^2 \phi\Big(T\log\frac{n}{x}\Big).
\end{equation}
We expand the square and swap the order of summation; since $\lambda_{K}(n)\in\R$, \eqref{eqn:ratio_3} equals
\begin{equation}
\label{eqn:ratio_4}
\begin{aligned}
&\sup_{\|\beta\|_2=1}\sum_{K,K'\in\mathfrak{F}_G(Q)}\beta(K)\overline{\beta(K')}\Big(\sum_{n\geq 1}\lambda_{K}(n)\lambda_{K'}(n)\phi\Big(T\log\frac{n}{x}\Big)\Big)\\
&\leq \sup_{\|\beta\|_2=1}\sum_{K,K'\in\mathfrak{F}_G(Q)}\frac{|\beta(K)|^2+|\beta(K')|^2}{2}\Big|\sum_{n\geq 1}\lambda_{K}(n)\lambda_{K'}(n)\phi\Big(T\log\frac{n}{x}\Big)\Big|\\
&=\sup_{\|\beta\|_2=1}\sum_{K'\in\mathfrak{F}_G(Q)}|\beta(K')|^2\sum_{K\in\mathfrak{F}_G(Q)}\Big|\sum_{n\geq 1}\lambda_{K}(n)\lambda_{K'}(n)\phi\Big(T\log\frac{n}{x}\Big)\Big|\\
&\leq \max_{K'\in\mathfrak{F}_G(Q)}\sum_{K\in\mathfrak{F}_G(Q)}\Big|\sum_{n\geq 1}\lambda_{K}(n)\lambda_{K'}(n)\phi\Big(T\log\frac{n}{x}\Big)\Big|.
\end{aligned}
\end{equation}
To estimate the inner sum in \eqref{eqn:ratio_4}, we apply \cref{lem:local_density} to the $\#\mathfrak{F}_G(Q)-\mathfrak{m}_{\mathfrak{F}_G}(Q)$ fields $K\in\mathfrak{F}_G(Q)$ such that $K\cap K'=\Q$, and we apply \cref{lem:trivial_bound_coeffs} to the remaining $\mathfrak{m}_{\mathfrak{F}_G}(Q)$ fields.  Thus, \eqref{eqn:ratio_4} is
\[
\ll_{m,\epsilon}\mathfrak{m}_{\mathfrak{F}_G}(Q)\frac{x(\log x)^{m^2-1}}{T}+\sqrt{x}T^{\frac{m^2}{4}}(\mathfrak{m}_{\mathfrak{F}_G}(Q)+(\#\mathfrak{F}_G(Q)-\mathfrak{m}_{\mathfrak{F}_G}(Q))Q^{\frac{m}{2}+\epsilon}T^{\epsilon}(\log x)^{B_m}).
\]
We have $1\leq \mathfrak{m}_{\mathfrak{F}_G}(Q)\leq\#\mathfrak{F}_G(Q)$, and $\#\mathfrak{F}_G(Q)\ll_{|G|} Q$ when $|G|\geq 5$ \cite[Proposition 1.3]{EV2}, $|G|=4$ \cite[Theorems 1 and 4]{MR2183288}, $|G|=3$ \cite{MR491593}, and $|G|=2$ (classical).  The theorem follows.
\end{proof}

\begin{corollary}
\label{cor:MVT_primes}
	Let $Q,T\geq 1$ and $\epsilon>0$.  If $y\geq Q^{m+2+\epsilon}T^{\frac{m^2}{2}+2+\epsilon}$ and $u\in[y,y^{12000}]$, then
	\[
	\sum_{K\in\mathfrak{F}_G(Q)}\int_{-T}^{T}\Big|\sum_{y< p \leq u}\frac{\lambda_{K}(p)\log p}{p^{1+it}}\Big|^2 dt\ll_{m,\epsilon} \mathfrak{m}_{\kF_G}(Q)(\log y)^{m^{2}-1}(\log u)^2.
	\] 
\end{corollary}
\begin{proof}
	Gallagher \cite[Theorem 1]{Gallagher} proved that if $\sum_{n=1}^{\infty}|c(n)|<\infty$, then
	\[
	\int_{-T}^{T}\Big|\sum_{n}c(n) n^{-it}\Big|^2 dt\ll T^2\int_0^{\infty}\Big|\sum_{n \in(x,xe^{1/T}]}c(n)\Big|^2\frac{dx}{x}.
	\]
	Let $b(n)$ be a function with finite support, and choose $c(n)=\lambda_{K}(n)b(n)$, in which case
	\[
	\sum_{K\in\mathfrak{F}_G(Q)}\int_{-T}^{T}\Big|\sum_{n\geq 1}\lambda_{K}(n)b(n) n^{-it}\Big|^2 dt\ll T^2\int_0^{\infty}\sum_{K\in\mathfrak{F}_G(Q)}\Big|\sum_{\substack{n \in(x,xe^{1/T}]}}\lambda_{K}(n)b(n)\Big|^2\frac{dx}{x}.
	\]
	We apply \cref{thm:pre_large_sieve} (with $\epsilon$ replaced by $\frac{\epsilon}{4}$) and $\mathfrak{m}_{\mathfrak{F}_G}(Q)\geq 1$ to bound the above display by
	\begin{multline*}
	\ll_{m,\epsilon} \mathfrak{m}_{\mathfrak{F}_G}(Q)\sum_{n=1}^{\infty}|b(n)|^2 n (\log en)^{m^{2}-1} (1+Q^{\frac{m}{2}+1+\frac{\epsilon}{4}}T^{\frac{m^2}{4}+1+\frac{\epsilon}{4}}n^{-\frac{1}{2}}(\log  en)^{B_m+1-m^2})\\
	\ll_{m,\epsilon} \mathfrak{m}_{\mathfrak{F}_G}(Q)\sum_{n=1}^{\infty}|b(n)|^2 n (\log en)^{m^{2}-1} (1+Q^{\frac{m}{2}+1+\frac{\epsilon}{4}}T^{\frac{m^2}{4}+1+\frac{\epsilon}{4}}n^{\frac{\epsilon}{12m^2+4\epsilon}-\frac{1}{2}}).
	\end{multline*}
	Let $y\geq Q^{m+2+\epsilon}T^{\frac{m^2}{2}+2+\epsilon}$ and $u\in[y,y^{12000}]$, and choose $b(n)=\frac{\log n}{n}$ when $n\in[y,u]$ is prime and $b(n)=0$ otherwise.  With these selections, the above display is
	\[
	\ll_{m,\epsilon}\sum_{K\in\mathfrak{F}_G(Q)}\int_{-T}^{T}\Big|\sum_{y< p \leq u}\frac{\lambda_{K}(p)\log p}{p^{1+it}}\Big|^2 dt\ll_{m}\mathfrak{m}_{\kF_G}(Q)(\log y)^{m^{2}-1}\sum_{p \leq u}\frac{(\log p)^2}{p}.
	\]
	Since $\sum_{p \leq u}\frac{(\log p)^2}{p}\ll(\log u)^2$ by partial summation and Mertens's theorem, the result follows.
\end{proof}

\section{Proof of \cref{thm:ZDE-Artin}}
\label{sec:ZDE}

In this section, we use \cref{cor:MVT_primes} to prove \cref{thm:ZDE-Artin}.  Our approach closely follows the work of Soundararajan and Thorner \cite{ST}.  Indeed, by the results in \cref{sec:Artin}, it follows that for each $K\in\mathfrak{F}_G$, the $L$-function $L(s,\chi_K)$ is in the class $\mathcal{S}(m)$ described in \cite[Subsections 1.1-1.4]{ST}, where $m=|G|-1$ (and the analytic conductor is $D_K e^{O(m)}$).  We will rely heavily on the results in \cite{ST}, though we require some minor modifications.

Let $K\in\mathfrak{F}_G(Q)$.  We have the Dirichlet series identity
\begin{equation}
\label{eqn:lambda_def}	
-\frac{L'}{L}(s,\chi_K) = \sum_{n}\frac{a_{K}(n) \Lambda(n)}{n^s}=\sum_{k=1}^{\infty}\sum_{p}\frac{\sum_{j=1}^{m} \alpha_{j,K}(p)^k\log p}{p^{ks}},\qquad \re(s)>1,
\end{equation}
where $\Lambda(n)$ is the usual von Mangoldt function.  Since $|\alpha_{j,K}(p)|\leq 1$ always, it follows that
\begin{equation}
	\label{eqn:mertens}
	\sum_{n=1}^{\infty}\frac{|a_{K}(n) \Lambda(n)|}{n^{1+\eta}}\leq m \sum_{n=1}^{\infty}\frac{\Lambda(n)}{n^{1+\eta}}\leq \frac{m}{\eta},\qquad \eta>0.
\end{equation}

By hypothesis, $\Lambda(s,\chi_K)$ is entire of order 1 and has the Hadamard product representation
\[
\Lambda(s,\chi_K)=e^{a_{K}+b_{K}s}\prod_{\rho}\Big(1-\frac{s}{\rho}\Big)e^{s/\rho},
\]
where $\rho$ ranges over the nontrivial zeros of $L(s,\chi_K)$.  Such zeros $\rho=\beta+i\gamma$ satisfy $0<\beta<1$.

\begin{lemma}
\label{lem:3.1}
	 If $0<\eta\leq 1$ and $t\in\R$, then
	\[
	\sum_{\rho}\frac{1+\eta-\beta}{|1+\eta+it-\rho|^2}\leq 2m\log D_K +m\log(2+|t|)+\frac{2n}{\eta}+O_m(1)
	\]
	and $\#\{\rho\colon |\rho-(1+it)|\leq\eta\}\leq 10\eta n\log  D_K +5\eta n\log(2+|t|)+O(m^2)$.
\end{lemma}
\begin{proof}
This is \cite[Lemma 3.1]{ST} applied to $L(s,\chi_K)$.
\end{proof}

\subsection{Detecting zeros} Recall $K \in \kF_G(Q)$ so $D_K \leq Q$. Let $k\geq 1$ be an integer, and let
\begin{equation}
\label{eqn:eta_range}
\frac{1}{\log QT}<\eta\leq \frac{1}{220m}.
\end{equation}
Let $s_0 = 1+\eta+i\tau$.  By \cite[(4.2)]{ST} and the fact that $L(s,\chi_K)$ is entire, we find that
\begin{equation}
	\label{eqn:4.2}
	\Big|\frac{(-1)^k}{k!}\Big(\frac{L'}{L}(s_0,\chi_K)\Big)^{(k)}-\sum_{|s_0-\rho|<200\eta}\frac{1}{(s_0-\rho)^{k+1}}\Big|\ll_m \frac{\log(QT)}{(200\eta)^k}.
\end{equation}

\begin{lemma}
	\label{lem:4.2}
	Let $\tau\in\R$ satisfy $|\tau|\leq T$, and let $\eta$ satisfy \eqref{eqn:eta_range}.  If $L(s,\chi_K)$ has a zero $\rho_0$ with $|1+i\tau-\rho_0|\leq\eta$ and $M>\lceil 2000\eta m \log(QT)+O_m(1)\rceil$, then there exists $k\in[M,2M]\cap\Z$ such that
	\[
	\Big|\sum_{\substack{\rho \\ |s_0-\rho|\leq200\eta}}\frac{1}{(s_0-\rho)^{k+1}}\Big|\geq\frac{1}{(100\eta)^{k+1}}.
	\]
\end{lemma}
\begin{proof}
This is \cite[Lemma 4.2]{ST} applied to $L(s,\chi_K)$.
\end{proof}

\begin{lemma}
\label{lem:prime_power_contribution}
	Let $K\in\mathfrak{F}_G(Q)$.  Let $\tau\in\R$ satisfy $|\tau|\leq T$, and let $\eta$ satisfy \eqref{eqn:eta_range}.  Let $M\geq 1$ and $k\in[M,2M]$ be integers, $N_0=\exp(M/(300\eta))$, and $N_1=\exp(40M/\eta)$.  We have that
	\[
	\Big|\frac{\eta^{k+1}}{k!}\Big(\frac{L'}{L}(s_0,\chi_K)\Big)^{(k)}\Big|\leq \eta^2\int_{N_0}^{N_1}\Big|\sum_{N_0\leq  p \leq u}\frac{\lambda_{K}(p)\log p}{p^{1+i\tau}}\Big|\frac{du}{u}+O_m\Big(\frac{\eta \log(QT)}{(110)^k}\Big).
	\]
	\end{lemma}

	\begin{proof}
We follow the proof of \cite[Lemma 4.3]{ST}.  Since $\eta>0$, one has
\[
\Big|\frac{\eta^{k+1}}{k!}\Big(\frac{L'}{L}(s_0,\chi_K)\Big)^{(k)}\Big|=\eta\Big|\sum_{n}\frac{a_{K}(n) \Lambda(n)}{n^{1+i\tau}}j_k(\eta\log n)\Big|,\qquad j_k(u) = e^{-u}\frac{u^k}{k!}.
\]
If $n \notin(N_0,N_1)$, then $j_k(\eta\log n)\leq  n^{-\eta/2}(110)^{-k}$ \cite[Proof of Lemma 4.3]{ST}.  So by \eqref{eqn:mertens},
\begin{equation}
\label{eqn:prime_power_contribution}
\Big|\sum_{n \notin[N_0,N_1]}\frac{a_{K}(n) \Lambda(n)}{n^{1+i\tau}}j_k(\eta\log n)\Big|\ll \frac{1}{(110)^k}\sum_{n}\frac{|a_{K}(n) \Lambda(n)|}{n^{1+\eta/2}}\ll_m \frac{\log(QT)}{(110)^k}.
\end{equation}
To handle the composite $n\in[N_0,N_1]$, note that since $(\log u)^k\leq k!u$ for all $k,u\geq 1$, we have
\[
j_k(\eta\log p^r) = p^{-r\eta}\frac{(\eta\log  p^r)^k}{k!}=p^{-r\eta}(2\eta)^k\frac{(\log  p^{\frac{r}{2}})^k}{k!}\leq (2\eta)^k  p^{r(\frac{1}{2}-\eta)}\leq\frac{p^{r(\frac{1}{2}-\eta)}}{(110)^k}.
\]
Since $|a_{K}(n)|\leq m$, it follows from \eqref{eqn:mertens} that
\begin{equation}
\label{eqn:prime_power_contribution'}
	\Big|\sum_{\substack{n \in[N_0,N_1] \\ \textup{$n$ composite}}}\frac{a_{K}(n) \Lambda(n)}{n^{1+i\tau}}j_k(\eta\log n)\Big|\ll_m \frac{1}{(110)^k}\sum_{p}\sum_{r\geq 2}\frac{\log  p }{p^{r(\frac{1}{2}+\eta)}}\ll_m \frac{\log(QT)}{(110)^k}.
\end{equation}

At a prime $p$, we have $a_K(p)\Lambda(p)=\lambda_K(p)\log p$.  Hence, by partial summation,
\begin{multline*}
	\Big|\sum_{p \in[N_0,N_1] }\frac{a_{K}(p) \Lambda(p)}{p^{1+\eta+i\tau}}\frac{(\eta \log  p )^k}{k!}\Big|\leq j_k(\eta\log N_1)\sum_{N_0\leq n\leq N_1}\frac{|\lambda_K(p)|\log p}{p}\\
	+\int_{N_0}^{N_1}\Big|\frac{d}{du}j_k(\eta\log u)\Big|\cdot\Big|\sum_{N_0\leq n\leq u}\frac{\lambda_K(p)\log p}{p^{1+i\tau}}\Big|\frac{du}{u}.
\end{multline*}
By our estimate for $j_k(\eta\log n)$ at the onset and \eqref{eqn:mertens}, the first term is $\ll m\log(QT)(110)^{-k}$.  Now $|\frac{d}{du}(j_k(\eta\log u))|=|-j_k(\eta\log u)+j_{k-1}(\eta\log u)|(\eta/u)\leq\eta/u$, so we conclude that
\begin{equation}
\label{eqn:prime_power_contribution_3}
\Big|\sum_{p \in[N_0,N_1] }\frac{a_{K}(p) \Lambda(p)}{p^{1+\eta+i\tau}}\frac{(\eta \log  p )^k}{k!}\Big|\leq \eta\int_{N_0}^{N_1}\Big|\sum_{N_0\leq p \leq u}\frac{\lambda_{K}(p)\log p}{p^{1+i\tau}}\Big|\frac{du}{u} +O_m\Big(\frac{\log(QT)}{(110)^k}\Big).
\end{equation}
The lemma follows once we combine the contributions from \eqref{eqn:prime_power_contribution}, \eqref{eqn:prime_power_contribution'}, and \eqref{eqn:prime_power_contribution_3}.
\end{proof}

\subsection{Counting zeros}

Our work in the previous subsection produces an upper bound for the count of zeros of $L(s,\chi_K)$ close to the line $\re(s)=1$.

\begin{lemma}
\label{lem:zde_upper}
Under the hypotheses of \cref{lem:prime_power_contribution}, if $M \geq 10^5 m^3\eta\log(QT)+O_{m}(1)$ with a sufficiently large implied constant,
	\begin{align*}
N_{K}(1-\tfrac{\eta}{2},T)&\ll_{|G|} (101)^{4M}M\eta^2 \int_{-T}^{T}\int_{N_0}^{N_1}\Big|\sum_{N_0\leq  p \leq u}\frac{\lambda_{K}(p)\log p}{p^{1+i\tau}}\Big|^2\frac{du}{u}d\tau.
\end{align*}
\end{lemma}
\begin{proof}
We mimic the proof of \cite[(4.5)]{ST} using \cref{lem:prime_power_contribution} instead of \cite[Lemma 4.3]{ST}.
\end{proof}

\begin{proof}[Proof of \cref{thm:ZDE-Artin}]
Choose $M = 10^5 m^3 \eta \log(QT)+O_{m}(1)$.  Recall that the range of $\eta$ is given by \eqref{eqn:eta_range}, $N_0=\exp(M/(300\eta))$, and $N_1=\exp(40M/\eta)$.  By \cref{lem:zde_upper}, the bound
	\begin{align}
		\label{eqn:estimate_0}
	\sum_{K\in\mathfrak{F}_G(Q)}N_{K}(1-\tfrac{\eta}{2},T)\ll_{m} (101)^{4M}M\eta^2 \int_{N_0}^{N_1}\sum_{K\in\mathfrak{F}_G(Q)}\int_{-T}^{T}\Big|\sum_{N_0\leq  p \leq u}\frac{\lambda_{K}(p)\log p}{p^{1+i\tau}}\Big|^2 \frac{d\tau du}{u}\hspace{-2.7mm}
	\end{align}
holds.  We apply \cref{cor:MVT_primes} with $y=N_0$.  This is valid; when $\epsilon=1$, we have $(10^5m^3)/300  \geq \max\{m+2+\epsilon,\frac{m^2}{2}+2+\epsilon\}$. By \cref{cor:MVT_primes}, \eqref{eqn:estimate_0} is $\ll_{m} \mathfrak{m}_{\kF_G}(Q)(101)^{4M}M^3 (\log QT)^{m^2}$.  Using our particular choices of $M$ and $\eta$ and writing $\sigma=1-\frac{\eta}{2}$, we conclude that
\begin{equation}
\label{eqn:almost_last_step}
\sum_{K\in\mathfrak{F}_G(Q)}N_{K}(\sigma,T)\ll_{m} \mathfrak{m}_{\kF_G}(Q) (QT)^{10^7 m^3(1-\sigma)}(\log QT)^{m^2},\quad 1-\frac{1}{440m}\leq\sigma<1-\frac{1}{2\log(QT)}.
\end{equation}
If $\sigma\geq 1-\frac{1}{2\log(QT)}$, then by the definition \eqref{eqn:NsigmaT} of $N_K(\sigma,T)$ and \eqref{eqn:almost_last_step}, we have that
\[
N_K(\sigma,T)\leq N_K(1-\tfrac{1}{2\log(QT)},T)\ll_{m}\mathfrak{m}_{\mathfrak{F}_G}(Q)(\log QT)^{m^2} \asymp \mathfrak{m}_{\kF_G}(Q) (QT)^{10^7 m^3(1-\sigma)}(\log QT)^{m^2}.
\]
On the other hand, $N_{K}(0,T)\ll_{m} T\log(QT)$ for all $K\in\mathfrak{F}_G(Q)$ \cite[Theorem 5.8]{IK}.  Thus, if $\sigma<1-\frac{1}{440m}$, then our estimate is trivial.  Since $|G|=m+1$, the theorem follows.
\end{proof}
 
\section{The error term in the Chebotarev density theorem}
\label{sec:avg_Chebotarev_error}

\cref{thm:main-Artin} follows from \cref{prop:FlexiError-Li,prop:large-ZFR}, the latter of which will use \cref{thm:ZDE-Artin}.  Let $G$ be a nontrivial finite group, and let $\mathfrak{F}_G$ be as in \cref{thm:ZDE-Artin}.  For $K\in\mathfrak{F}_G$, let $\Delta_{K}:[3,\infty)\to[0,\frac{1}{2}]$ be a function such that $\zeta_K(s)/\zeta(s)\neq 0$ in the region $\re(s)\geq 1-\Delta_K(|\im(s)|+3)$, and define
\[
\eta_K(x):=\inf_{t\geq 3}(\Delta_K(t)\log x+\log t).
\]

Our first of two key propositions for proving \cref{thm:main-Artin} is an unconditional effective form of the Chebotarev density theorem that depends only on $\eta_K(x)$.

\begin{proposition} \label{prop:FlexiError-Li}
Let $K\in\mathfrak{F}_G$, $C\subseteq G$ be a conjugacy class.  If $x\geq \max\{(\log D_K)^{4},5000\}$, then
	\[
	\Big| \pi_C(x,K/\Q) - \frac{|C|}{|G|} \int_2^x\frac{dt}{\log t} \Big| \ll_{|G|} \frac{x}{\log x}(e^{-\frac{1}{3}\sqrt{\log x}}+\Delta^{\frac{1}{2}}+e^{-\frac{1}{4}\eta_K(x)}\log D_K)
	\]
	with
	\[
	\Delta = x^{-\frac{1}{4}}+\min\Big\{ \frac{1}{16}, 8 e^{-\frac{\eta_{K}(x)}{4}}\Big\} + \min\Big\{ \frac{1}{16}, 8 e^{-\frac{\sqrt{\log x}}{6}} \Big\}.
	\]
	\end{proposition}
The ideas are classical but require careful execution.  We borrow from the analysis in \cite[Sections 2 and 4]{TZ3} and refer there for some of the details.  We first introduce a weight function.

\begin{lemma}\label{lem:WeightChoice}

Let $x \geq 5000$, $t\in\R$, and $\Delta$ be as in \cref{prop:FlexiError-Li}.  There exists a continuous function $f(t) = f_{x,\Delta}(t)$ of a real variable $t$ such that:
\begin{enumerate}[(i)]
	\item $0 \leq f(t) \leq 1$ for all $t \in \R$, and $f(t) = 1$ for $\tfrac{1}{2} \leq t \leq 1$.
	\item The support of $f$ is contained in the interval $[\tfrac{1}{2} - \frac{\Delta}{\log x}, 1 +  \frac{\Delta}{\log x}]$. 
	\item Its Laplace transform $F(z) = \int_{\R} f(t) e^{-zt}dt$ is entire and is given by
			\begin{equation*}	
				F(z) = e^{-(1+ \frac{\Delta}{\log x})z} \cdot \Big( \frac{1-e^{(\frac{1}{2}+\frac{\Delta}{\log x})z}}{-z} \Big) \Big( \frac{1-e^{\frac{\Delta z}{2  \log x}}}{- \frac{\Delta z}{2\log x}} \Big)^{2}.
			\end{equation*}
	\item Let $s = \sigma + i t, \sigma > 0,$ and $t \in \R$. Then $|F(-s\log x)|\leq 
		e^{\sigma \Delta} x^{\sigma} \min\{ 1,  \frac{1 + x^{-\sigma/2} }{|s|\log x}  ( \frac{4}{\Delta|s|} )^{2} \}$.
	\item We have $\frac{1}{2} < F(0) < \frac{3}{4}$ and $F(-\log x)  = \frac{x}{\log x} + O(\frac{\Delta x  + x^{1/2}}{\log x} )$.
\end{enumerate}
\end{lemma}	
\begin{proof}
	If $x\geq 5000$, then $0<\Delta<\frac{1}{4}$.  Thus, the result follows from \cite[Lemma 2.2]{TZ3} with $\ell=2$.
\end{proof}

\begin{proof}[Proof of \cref{prop:FlexiError-Li}]
Note that $\log D_K\leq x^{1/4}$ by hypothesis.  Let $\Delta$ be as in \cref{prop:FlexiError-Li}, and let $f(\, \cdot \,) = f_{x,\Delta}( \, \cdot \, )$ be as in \cref{lem:WeightChoice}.  Let $F(\, \cdot \,)=F_{x,\Delta}(\, \cdot \,)$ be the Mellin transform of $f_{x,\Delta}( \, \cdot \, )$.  If $a,b,c>0$, then $1/(a+b+c)<\min\{a^{-1},b^{-1},c^{-1}\}$, so $\Delta^{-2}\ll \min\{x^{1/2},e^{\eta_K(x)/2},e^{\sqrt{\log x}/3}\}$.

Let $K\in\mathfrak{F}_G$, and let $C\subseteq G$ be a conjugacy class.  Let $g\in C$, $H$ be the cyclic group $\langle g\rangle$, and $K^H$ be fixed field of $H$.  Let $F_1(s)=\frac{x^s}{s}$ and $F_2(s)=F(-s\log x)\log x$, and define for $j\in\{1,2\}$
\[
\psi_{j}(x)=\frac{1}{|H|}\sum_{\chi\in\hat{H}}\frac{\bar{\chi}(g)}{2\pi i}\int_{2-i\infty}^{2+i\infty}-\frac{L'}{L}(s,\chi,K^H)F_j(s)ds,
\]
where $L(s,\chi,K^H)$ is the Hecke $L$-function attached to the one-dimensional character $\chi\in\widehat{H}$.  By \cite[Lemmata 2.1, 2.3, and 5.2]{TZ3}, we have
\begin{equation}
\label{eqn:ReductionToAbelian}
	\Big|\pi_C(x,K/\Q)-\frac{|C|}{|G|}|H|\Big(\frac{\psi_1(x)}{\log x}+\int_{\sqrt{x}}^{x}\frac{\psi_1(t)}{t(\log t)^2}dt\Big)\Big|\ll_{|G|}\sqrt{x} +\log D_K\ll \sqrt{x}
\end{equation}
and
\begin{equation}
\label{eqn:ReductionToAbelian2}
|\psi_1(x)-\psi_2(x)|\ll_{|G|} \sqrt{x}+\Delta x.
\end{equation}
Let $\rho$ be a nontrivial zero of $L(s,\chi,K^H)$. By \cref{lem:WeightChoice} and  \cite[Lemmata 4.3, and 4.4]{TZ3}, we have
		\begin{equation}
		\label{eqn:Artin-ContourShift}
		\begin{aligned}
		|H|\frac{\psi_{2}(x)}{\log x} &= F(-\log x)-\sum_{\chi\in\widehat{H}}\bar{\chi}(C_H)\sum_{|\rho|>\frac{1}{4}}F(-\rho\log x)+O_{|G|}\Big(\frac{\Delta^{-2}\log D_K}{x^{1/4}\log x}+\frac{1}{\log x}\Big)\\
		&=\frac{x}{\log x}-\sum_{\chi\in\widehat{H}}\bar{\chi}(C_H)\sum_{|\rho|>\frac{1}{4}}F(-\rho\log x)+O_{|G|}\Big(\frac{\sqrt{x}+\Delta x}{\log x}\Big).
		\end{aligned}\hspace{-2mm}
		\end{equation}
		Since $\prod_{\chi\in\widehat{H}}L(s,\chi,K^H)=\zeta_K(s)= \zeta(s)L(s,\chi_K)$, we can partition the zeros $\rho=\beta+i\gamma$ in \eqref{eqn:Artin-ContourShift} according to whether $\zeta(\rho)=0$ or $L(\rho,\chi_K)=0$ because of \cref{thm:Aramata_Brauer}. If $\zeta(\beta+i\gamma)=0$, then $|\gamma|>14$ and $\beta\leq 1-1/(6\log|\gamma|)$ \cite{Kadiri}.  By \cref{lem:WeightChoice}(iv) and our choice of $\Delta$, we have
		\begin{align}
		\label{eqn:Artin-isolatePi}
		\Big|\sum_{\chi\in\widehat{H}}\bar{\chi}(C_H)\sum_{|\rho|>\frac{1}{4}}F(-\rho\log x)\Big| &\ll \sum_{\substack{|\rho|>\frac{1}{4} \\ \zeta(\rho)=0}}|F(-\rho\log x)|+\sum_{\substack{|\rho|>\frac{1}{4} \\ L(\rho,\chi_K)=0}}|F(-\rho\log x)|\notag\\
		&\ll \sum_{\substack{|\gamma|>14 \\ \zeta(\beta+i\gamma)=0}}\frac{x^{\beta}\Delta^{-2}}{|\gamma|^3}+\sum_{\substack{|\beta+i\gamma|>\frac{1}{4} \\ L(\beta+i\gamma,\chi_K)=0}}\frac{x^{\beta}\Delta^{-2}}{(|\gamma|+3)^3}\\
		&\ll xe^{\frac{\sqrt{\log x}}{3}}\sum_{\substack{|\gamma|>14 \\ \zeta(\beta+i\gamma)=0}}\frac{x^{-\frac{1}{6\log|\gamma|}}}{|\gamma|^3}+e^{\frac{\eta_K(x)}{2}}\sum_{\substack{|\beta+i\gamma|>\frac{1}{4} \\ L(\beta+i\gamma,\chi_K)=0}}\frac{x^{\beta}}{(|\gamma|+3)^3}.\notag
		\end{align}
		If $t\in\R$ and $M$ is a number field, then there are $\ll_{[M:\Q]}\log(D_M(|t|+3))$ nontrivial zeros $\beta+i\gamma$ of $\zeta_M(s)$ with $|\gamma-t|<1$ \cite[Lemma 5.4]{LO}.  Therefore, by Laplace's method, we find that
		\begin{equation}
		\label{eqn:sum_over_zeta_zeros}
		xe^{\frac{\sqrt{\log x}}{3}}\sum_{\substack{|\gamma|>14 \\ \zeta(\beta+i\gamma)=0}}\frac{x^{-\frac{1}{6\log|\gamma|}}}{|\gamma|^3}\ll xe^{\frac{\sqrt{\log x}}{3}}\int_{14}^{\infty}x^{-\frac{1}{6\log t}}t^{-3}(\log t)dt\ll xe^{-\frac{5}{6}\sqrt{\log x}}.
		\end{equation}
		If $L(\beta+i\gamma,\chi_K)=0$, then $x^{\beta}(|\gamma|+3)^{-1}\leq x e^{-\eta_K(x)}$ and
		\begin{equation}
		\label{eqn:sum_over_zeta_zeros2}
		e^{\frac{\eta_K(x)}{2}}\sum_{\substack{|\beta+i\gamma|>\frac{1}{4} \\ L(\beta+i\gamma,\chi_K)=0}}\frac{x^{\beta}}{(|\gamma|+3)^3}\ll xe^{-\frac{\eta_K(x)}{2}}\sum_{\substack{|\beta+i\gamma|>\frac{1}{4} \\ L(\beta+i\gamma,\chi_K)=0}}\frac{1}{(|\gamma|+3)^2}\ll_{|G|} x e^{-\frac{\eta_K(x)}{2}}\log D_K.
		\end{equation}

We deduce from \eqref{eqn:Artin-ContourShift}-\eqref{eqn:sum_over_zeta_zeros2} that $\psi_2(x)=x/|H|+O(x(e^{-\frac{5}{6}\sqrt{\log x}}+\Delta+e^{-\frac{1}{2}\eta_K(x)}\log D_K))$.  By \eqref{eqn:ReductionToAbelian2}, the same asymptotic holds for $\psi_1(x)$.  In passing from $\psi_1(x)$ to $\pi_C(x,K/\Q)$ in \eqref{eqn:ReductionToAbelian}, we handle the integral on $[\sqrt{x},x]$ using the fact that $\eta_K(y)$ is increasing and $\eta_K(y^{1/2})\geq \frac{1}{2}\eta_K(y)$.
\end{proof}

It remains to produce a suitable bound for $\max\{e^{-\eta_K(x)/4}\log D_K,e^{-\eta_K(x)/8}\}\leq e^{-\eta_K(x)/8}\log D_K$.

\begin{proposition} \label{prop:large-ZFR}
	Let $0<\epsilon<1$ and $Q\geq 3$. For all except $O_{|G|,\epsilon}( \mathfrak{m}_{\mathfrak{F}_G}(Q) Q^{\epsilon})$ number fields $K \in \mathfrak{F}_G(Q)$,  if $x \geq (\log D_K)^{10^9|G|^3/\epsilon}$, then
	\[
	e^{-\frac{\eta_K(x)}{8}}\log  D_K\leq \exp\Big(-\frac{1}{30}\Big(\frac{\log x}{|G|}\Big)^{\frac{1}{2}}\Big).
	\]
\end{proposition}

\begin{proof}
Let $\epsilon_0 = \epsilon/(2|G|^2+2)$ and $\delta = 10^{-9}|G|^{-3}\epsilon$.  For $2 \leq j \leq Q^{\epsilon_0}+1$, we use \cref{thm:ZDE-Artin} with
	\[
	 T =  T_j := e^j -3 \quad \text{and} \quad \sigma = \sigma_j := 1 - \frac{20\delta \log Q}{\log Q +  \log(T_j+3)}.
	 \]
	 By iteratively applying \cref{thm:ZDE-Artin}, we throw  out $O_{|G|,\epsilon}(\mathfrak{m}_{\mathfrak{F}_G}(Q) Q^{ 10^7|G|^3\cdot 20\delta+|G|^2 \epsilon_0})$ exceptions at most $O_{\epsilon}(Q^{\epsilon_0})$ times. This dyadically builds a zero-free region aside from $O_{|G|,\epsilon}(\mathfrak{m}_{\mathfrak{F}_G}(Q) Q^{\epsilon})$ exceptional fields. Thus, for all except at most $O_{|G|,\epsilon}(\mathfrak{m}_{\mathfrak{F}_G}(Q) Q^{\epsilon})$ of the fields $K \in \mathfrak{F}_G(Q)$, the $L$-function $L(s,\chi_K)$ does not vanish in the region
	\[
	\Re(s) > 1 - \frac{20\delta \log Q}{\log Q +  \log(|\Im(s)|+3)},  \qquad |\Im(s)| \leq \exp(Q^{\epsilon_0}).
	\]
	It follows from \cite[Theorem 1.1]{Lee} that $\zeta_{K}(s)$, hence $L(s,\chi_K)$, does not vanish in the region
	 \[
	 \Re(s) > 1 - \frac{1}{13(\log D_K  + [K:\Q]\log(|\Im(s)| +3))}, \qquad |\Im(s)| \geq 1. 
	 \]	
	Note that $|G|=[K:\Q]$ since $K/\Q$ is Galois.  Since $D_K\leq Q$, it follows that for all $K\in\mathfrak{F}_G(Q)$ with $O_{|G|,\epsilon}(\mathfrak{m}_{\mathfrak{F}_G}(Q)Q^{\epsilon})$ exceptions, we have that $L(s,\chi_K)\neq 0$ in the region
	\begin{equation}
	\label{eqn:zfr}
	\re(s)\geq 1-\Delta_K(|\im(s)|+3),\qquad \Delta_{K}(t) = \begin{cases}
 		\frac{20\delta \log D_K}{\log D_K +  \log t } & \mbox{if $3 \leq t \leq \exp(D_K^{\epsilon_0})$,}\\
 		\frac{1}{13(\log D_K + |G|\log t)} & \mbox{if $t > \exp(D_K^{\epsilon_0})$.}
	 \end{cases}
	\end{equation}
	
	Let $K\in\mathfrak{F}_G(Q)$ be such that $L(s,\chi_K)\neq 0$ in the region \eqref{eqn:zfr} and $D_K$ is sufficiently large with respect to $|G|$ and $\epsilon$. By the change of variables $t=e^u$, it follows that if $x \geq 3$, then
	\begin{equation}
	\eta_{K}(x) \geq \min\Big\{ \inf_{u \geq D_K^{\epsilon_0}} \phi_1(u,x),  \inf_{0 \leq u \leq D_K^{\epsilon_0}} \phi_2(u,x) \Big\},
	\label{eqn:large-ZFR-eta}
	\end{equation}
	where
	\begin{equation*}
	\begin{aligned}
	\phi_1(u,x)	= \frac{\log x}{13(\log D_K + |G| u)}  + u,
	\qquad
	\phi_2(u,x) = \frac{20\delta (\log D_K)(\log x)}{\log D_K+ u} + u.
	\end{aligned}
	\end{equation*}
	
	The global minimum of $\phi_1(u,x)$ for $u>-\frac{1}{|G|}\log D_{K}$ is attained at
	\[
	u=u_1 :=  \Big(\frac{\log x}{13|G|}\Big)^{1/2} - \frac{\log D_{K}}{|G|}.
	\]
	Thus, $\phi_1(u,x)$ is minimized for $u\in [D_{K}^{\epsilon_0},\infty)$ at $u = \max\{ u_1, D_{K}^{\epsilon_0}\}$. If $u_1 \geq D_K^{\epsilon_0}$, then
	\[
	\phi_1(u_1,x)=\Big(\frac{\log x}{13|G|}\Big)^{\frac{1}{2}}+u_1\geq \Big(\frac{\log x}{13|G|}\Big)^{\frac{1}{2}}+D_K^{\epsilon_0}.
	\]
	If $u_1 < D_K^{\epsilon_0}$ and $D_K$ is sufficiently large with respect to $|G|$ and $\epsilon$, then
	\[
	\phi_1(D_K^{\epsilon_0},x) > D_K^{\epsilon_0} > \frac{199}{200}u_1+\frac{D_K^{\epsilon_0}}{200}=\frac{199}{200}\Big(\frac{\log x}{13|G|}\Big)^{\frac{1}{2}}-\frac{199\log D_K}{200|G|}+\frac{D_K^{\epsilon_0}}{200}>\frac{199}{200}\Big(\frac{\log x}{13|G|}\Big)^{\frac{1}{2}}+\frac{D_K^{\epsilon_0}}{400}.
	\] 
	Combining these observations, we find that for $x\geq 3$,
	\begin{equation}
	\inf_{u \geq D_{K}^{\epsilon_0}} \phi_1(u,x) >\frac{199}{200}\Big(\frac{\log x}{13|G|}\Big)^{\frac{1}{2}}+\frac{D_K^{\epsilon_0}}{400}.
	\label{eqn:large-ZFR-phi1}
	\end{equation}
	Similarly, the global minimum of $\phi_2(u,x)$ for $u\in(-\log D_{K},\infty)$ is attained at
	\[
	u = u_2 :=  (20\delta (\log D_{K}) \log x)^{\frac{1}{2}} - \log D_{K}.
	\]
	For $u\in[0,D_{K}^{\epsilon_0}]$, the function $\phi_2(u,x)$ achieves its minimum at $u = 0, u_2,$ or $D_{K}^{\epsilon_0}$. Observe that $\phi_2(0,x) = 20\delta \log x$. Also, note that $u_2 \geq 0$ if and only if $\log x \geq \frac{\log D_{K}}{20\delta}$, in which case
	\[
	\phi_2(u_2,x) = 2(20\delta(\log D_K)\log x)^{\frac{1}{2}}-\log D_K \geq  (20\delta (\log D_{K}) \log x)^{\frac{1}{2}}.
	\]
	If $u_2  \geq 0$, then $\phi_2(u_2,x) \leq \phi_2(D_K^{\epsilon_0},x)$ since $u_2$ is the global minimum. Thus, if $x \geq 3$, then
	\begin{equation}
		\inf_{0 \leq u \leq D_{K}^{\epsilon_0}} \phi_2(u,x)  
		\geq \min\{20\delta \log x , (20\delta (\log D_{K}) \log x)^{\frac{1}{2}}\}.
	\label{eqn:large-ZFR-phi2}
	\end{equation}
	We deduce from \eqref{eqn:large-ZFR-eta}, \eqref{eqn:large-ZFR-phi1}, and \eqref{eqn:large-ZFR-phi2} that if $D_K$ is large with respect to $|G|$ and $\epsilon$, then
	\[
	e^{-\frac{\eta_K(x)}{8}}\log D_K\leq \exp\Big(-\frac{1}{8}\min\Big\{\frac{199}{200}\Big(\frac{\log x}{13|G|}\Big)^{\frac{1}{2}}+\frac{D_K^{\epsilon_0}}{400},20\delta\log x,(20\delta(\log D_K)\log x)^{\frac{1}{2}}\Big\}\Big)\log D_K.
	\]
	If we additionally assume that $x\geq (\log D_K)^{1/\delta}$, then the  proposition follows.
	\end{proof}

\bibliographystyle{abbrv}
\bibliography{ZDE-DedekindZeta_REVISION}

\end{document}